\documentclass{commutativealgebra}

\title{Amalgamated algebras along an ideal}
\authorone{Marco D'Anna}
\addressone{Dipartimento di Matematica e Informatica, Universit\`a di Catania}
\countryone{Italy}
\emailone{mdanna@dmi.unict.it}

\authortwo{Carmelo Antonio Finocchiaro}
\addresstwo{Dipartimento di Matematica, Universit\`a degli Studi ``Roma Tre''}
\countrytwo{Italy}
\emailtwo{carmelo@mat.uniroma3.it}

\authorthree{Marco Fontana}
\addressthree{Dipartimento di Matematica, Universit\`a degli Studi ``Roma Tre''}
\countrythree{Italy}
\emailthree{fontana@mat.uniroma3.it}

\researchsupported{ The first and the third authors were partially supported by MIUR- PRIN grants}

\abstract{{Let $f:A \rightarrow B$ be a ring homomorphism and  $J$   an ideal of $B$. In this paper, we  initiate
a systematic study of a new ring construction called the ``amalgamation of $A$ with $B$ along $J$   with respect
to $f$''.  This construction finds its roots in a paper by J.L. Dorroh appeared in 1932 and   provides a general
frame for studying  the amalgamated duplication of a ring   along an ideal, introduced and studied by D'Anna and
Fontana in 2007, and  other classical constructions such  as the $A+ XB[X]$ and $A+  XB[\![X]\!]$  constructions,
the CPI-extensions of Boisen and Sheldon,  the $D+M$ constructions and the Nagata's idealization.}}



\acknowledgments{}
\usepackage{amssymb,amsmath,amsthm,amscd}

\usepackage{color}   
\newcommand{\ec}{\color{black}}%
\newfont{\rams}{msbm10 scaled\magstep1}
\newfont{\ramss}{msbm10 scaled\magstep0}
\newfont{\iams}{msbm10}
\newfont{\gotic}{eufm10 scaled\magstep1}
\newfont{\bellap}{eusm10 scaled\magstep1}

\newcommand{\inte}{\mbox{\rams \symbol{'132}}}

\newcommand{\q}{{J}}

\newcommand{\aaa}{{I}}

\newcommand{\da}{{A\!\Join^f\!\!J}}

\newcommand{\Jac}{{\rm Jac}}%

\newcommand{\Ker}{{\rm Ker}}
\newcommand{\tot}{{\rm Tot}}

\newcommand{\z}{{\ldots}}
\newcommand{\w}{{\setminus}}

\newcommand{\ude}{\mbox{\rm{id}}} %
\newcommand{\udes}{\mbox{\rm{\tiny id}}} %
\swapnumbers

\firstpage{1}                                %

\begin{document}


\section{Introduction}

 Let $A$ and $B$ be commutative rings with unity, let $J$ ba an ideal of $B$ and let $f:A\longrightarrow B$ be a
ring homomorphism. In this setting, we can define the following subring of $A\times B$:
$$
A\Join^f\! J=\{(a,f(a)+j) \mid  a \in A, \ j \in J \}
$$
called {\it the amalgamation of $A$ with $B$ along $J$ with respect to $f$}. This construction is a generalization
of the amalgamated duplication of a ring along an ideal (introduced and studied in \cite{do1} and \cite{do2}).
Moreover, other classical constructions (such as the $A+XB[X]$ construction, the $D+M$ construction and the
Nagata's idealization) can be studied as particular cases of the amalgamation.

On the other hand, the amalgamation $A\Join^f\! J$ is related to
a construction proposed by D.D. Anderson in  \cite{a-06} and motivated by a classical construction due to Dorroh
\cite{do1}, concerning the embedding {of} a ring without identity in a ring with identity.

The level of generality that we have choosen is due to the fact that the amalgamation can be studied in the frame
of  pullback constructions. This point of view allows us to provide easily an ample description of the properties
of $\da$, in connection with the properties of $A$, $J$ and $f$.

In this paper,  we  begin a study of  the basic properties of $\da$. In particular, in Section 2, we present all
the constructions cited above as particular cases of the amalgamation. Moreover, we show that the CPI extensions
(in the sense of Boisen and Sheldon \cite{bo}) are related to amalgamations of a special type and we compare
Nagata's idealization with the amalgamation. In Section 3, we consider the iteration of the amalgamation process,
giving some geometrical applications of it.

In the last two sections, we show that the amalgamation can be realized as a pullback and we characterize those
pullbacks that arise from an amalgamation (Proposition \ref{fibret}. {Finally} we apply these results to
study the basic algebraic properties of the amalgamation, with particular attention to the finiteness conditions.



\ec\section{The genesis}

Let $A$ be a commutative ring with identity and let  $\mathcal{R}$   be a  ring without identity which is an
$A$-module. Following the construction described  by D.D. Anderson in  \cite{a-06},  we can define a
multiplicative structure  in the $A$--module  $A\oplus \mathcal{R}$,  by setting $(a,x)(a',x')
 := (aa',a x'+a'x+xx')$, for all    $a, a' \in A$
and $x, x' \in \mathcal{R}$.   We denote by $A {\boldsymbol {\dot{\oplus}}}\mathcal{R}  $ the
direct sum  $A\!\oplus\! \mathcal{R} $ endowed  also with the multiplication
defined above.
\smallskip

The following properties are easy to   check.

\begin{lemma} { \rm{\cite[Theorem 2.1]{a-06}}} \label{le:1} With the notation
introduced above, we have:
\begin{itemize}
  \item[\rm (1)] $A {\boldsymbol {\dot {\oplus}}}   \mathcal{R} $ is a ring with identity $(1, 0)$, which has an $A$--algebra structure induced by the canonical ring embedding $ \iota_A:  A \hookrightarrow  A {\boldsymbol {\dot {\oplus}}}   \mathcal{R}$,  defined by    $a \mapsto (a, 0)$  for all $a \in A$.
    \item[\rm (2)] If we identify $\mathcal{R}$ with  its canonical image $(0) \times \mathcal{R} $
  under the canonical embedding
  $\iota_{\mathcal{R}}: \mathcal{R}  \hookrightarrow A {\boldsymbol {\dot {\oplus}}}   \mathcal{R} $,
  defined by  $x \mapsto (0, x)$, for all $x \in  \mathcal{R} $,
  then $  \mathcal{R} $ becomes an ideal in $A {\boldsymbol {\dot {\oplus}}}    \mathcal{R} $.

 \item[\rm(3)] If we identify $A$ with  $A\times (0)$ (respectively,
 $  \mathcal{R} $ with $(0) \times  \mathcal{R}$) inside $A {\boldsymbol {\dot {\oplus}}}    \mathcal{R} $,
 then the ring $A {\boldsymbol {\dot {\oplus}}}   \mathcal{R} $
 is an $A$--module generated by $(1, 0)$ and $ \mathcal{R} $, i.e.,
    $A(1,0) +   \mathcal{R} =  A {\boldsymbol {\dot {\oplus}}}    \mathcal{R} $. Moreover,
    if $p_A: A {\boldsymbol {\dot {\oplus}}}    \mathcal{R} \twoheadrightarrow A$ is the canonical
projection  (defined by $(a, x) \mapsto a$ for all $a\in A$ and $x \in   \mathcal{R}$),   then
$$
0 \rightarrow \mathcal{R}  \xrightarrow { \iota_{_{\!\mathcal{R} }}} A {\boldsymbol {\dot {\oplus}}}   \mathcal{R}
\xrightarrow {p_{_{\!A}}} A \rightarrow 0$$  is a splitting {exact} sequence of $A$--modules.  \hfill $\Box$
\end{itemize}
\end{lemma}
\noindent
\begin{remark} (1)
 The previous construction  takes its roots in the classical construction, introduced   by Dorroh \cite{do1} in 1932, for embedding a ring   (with or without identity, possibly without regular elements) in a ring with identity (see also Jacobson \cite{j}, page 155). For completeness, we recall Dorroh's construction starting with a case which is not the motivating one, but that leads naturally  to the relevant one (Case 2).

{\bf Case 1}. Let $R$ be a commutative ring (with or without identity) and let ${\rm Tot}(R)$ be its total ring of fractions, i.e., $\tot(R) := N^{-1} R$, where $N$ is the set of regular elements of $R$.  If we assume that $R$ has a regular element $r$, then it is easy to see that   $R \subseteq  {\rm Tot}(R)$, and ${\rm Tot}(R)$ has identity $1:=\frac{r}{r}$,  even if $R$ does not.   In this situation we can consider $R[1]:=\{x+ m\cdot 1  \mid x\in R, m \in\mathbb Z\}$.   Obviously, if $ R$ has an identity,  then $R=R[1]$; otherwise, we have that $R[1]$ is a commutative ring with identity,   which contains properly $R$ and it  is the smallest subring of $\tot(R)$ containing $R$ and $1$.  It is easy to see that:
\begin{enumerate}
 \item[(a)]
 $R$ and $R[1]$ have the same  characteristic,
 \item[(b)]    $R$ is an ideal of $R[1]$ \,  and
 \item[(c)]   if $R\subsetneq R[1]$, then   the quotient-ring  $R[1]/R$ is canonically isomorphic to   $\mathbb Z/n\mathbb Z$,  where $n \ (\geq 0)$ is the characteristic of $R[1]$ (or, equivalently, of $R$).
\end{enumerate}

{\bf Case 2}. Let $R$ be a commutative ring (with or without identity) and, possibly, without regular elements. In
this situation, we possibly have   $R={\rm Tot}(R)$, so we cannot perform the previous construction. Following
Dorroh's ideas, we can consider in any case $R$ as a $\mathbb Z$-module and, with the notation introduced at the
beginning of this section,  we can construct the ring $\mathbb Z {\boldsymbol {\dot {\oplus}}}  R$, that we denote
by ${\mathtt {Dh}}(R)$ in Dorroh's honour. Note that ${\mathtt {Dh}}(R)$ is a commutative ring with identity
$1_{{\mathtt {Dh}}(R)} :=(1,0)$.   If we identify, as usual, $R$ with its canonical image in ${\mathtt {Dh}}(R)$,
then $R$ is an ideal of ${\mathtt {Dh}}(R)$ and ${\mathtt {Dh}}(R)$ has a kind of minimal property over $R$, since
${\mathtt {Dh}}(R)=\mathbb Z(1,0)+R$. Moreover, the quotient-ring ${\mathtt {Dh}}(R)/R$ is naturally isomorphic to
$\mathbb Z$.

On the bad side, note that if $R$ has an identity $1_R$, then the canonical embedding of $R$  into ${\mathtt
{Dh}}(R)$ (defined by $x \mapsto (0, x)$ for all $x \in R)$ does not preserve the identity, since $(0,1_R)\neq
1_{{\rm Dh}(R)}$. Moreover, in  any case (whenever $R$ is  a ring with or without identity), the canonical
embedding $R \hookrightarrow {\mathtt {Dh}}(R)$ may not preserve the characteristic.

In order to overcome this difficult, in 1935, Dorroh \cite{do2} gave a variation of  the previous  construction.
More precisely, if $R$ has positive characterisitic $n$, {{then}} $R$ can be considered  as  a $\mathbb Z/n\mathbb
Z$-module, so ${\mathtt {Dh}}_n(R):= \left(\mathbb Z/n\mathbb Z\right) {\boldsymbol {\dot {\oplus}}}  R$ is a ring
with identity, ha\-ving  characteristic $n$. Moreover, as above, ${\mathtt {Dh}}_n(R)=\left(\mathbb Z/n\mathbb
Z\right)(1,0)+R$ and ${\mathtt {Dh}}_n(R)/R$ is canonically isomorphic to $\mathbb Z/n\mathbb Z$.

\smallskip

(2) Note that a general version of the Dorroh's construction (previous Case 2) was  considered  in 1974  by Shores
\cite[Definition 6.3]{sh} for constructing examples   of local commutative rings with arbitrarily large Loewy
length.   We are indebted to L. Salce  for pointing out to us  that
 the amalgamated duplication of a ring along an ideal \cite{d'a-f-1} can also  be
 viewed as a special case of Shores construction (cf. also \cite{sa}).
Moreover, before Shores,  Corner in 1969 \cite{co}, for studying endomorphisms rings of Abelian groups, considered
a similar construction called ``split extension of a ring by an ideal''.
\end{remark}

\smallskip

A natural situation in which we can apply the previous general construction  (Lemma \ref{le:1})  is the following.
Let $f:A\rightarrow B$ be a ring homomorphism and let $\q$ be an ideal of $B$.  Note that $f$ induces on $\q$  a
natural structure of $A$--module by setting   $a\!\cdot\! j := f(a)j$,  for all $a \in A$ and $j \in \q$. Then, we
can consider $ {A\boldsymbol {\dot {\oplus}}}  \q$.

The following properties, except (2) that is easy to verify,  follow from Lemma \ref{le:1}.

\begin{lemma} \label{R_E}   With the notation
introduced above, we have:
\begin{itemize}
  \item[\rm(1)] $A {\boldsymbol {\dot {\oplus}}}  \q$ is a ring.
    \item[\rm (2)] The map $f^{\Join}: A {\boldsymbol {\dot {\oplus}}} \q
\rightarrow A \times B$,
    defined by  $(a, j)\mapsto (a, f(a)+j)$ for all $a \in A$ and $j \in J$,  is an injective ring
    homomorphism.
    \item[\rm(3)] The map $\iota_A: A  \rightarrow A
{\boldsymbol{\dot {\oplus}}} \q$ (respectively, $\iota_\q: \q  \rightarrow A
{\boldsymbol{\dot {\oplus}}} \q$),
    defined by $a\mapsto (a,0)$  for all $a \in A$   (respectively,  by $ j \mapsto (0, j)$   for all $j \in J$),   is an injective ring
homomorphism (respectively,  an injective $A$--module
homomorphism).   If we identify  $A$ with $\iota_A(A)$ (respectively,  $\q$ with $\iota_\q(\q)$), then the ring  $A {\boldsymbol {\dot {\oplus}}}  \q$ coincides with $A+\q$.

   \item[\rm(4)]  Let $p_A :  A {\boldsymbol {\dot {\oplus}}}  \q \rightarrow A $ be the canonical projection  (defined by $(a ,j) \mapsto a$ for all $a \in A$ and $j \in J$), then the following is a split exact sequence of $A$--modules:
   $$
   0 \rightarrow \q  \xrightarrow{\iota_J  }A {\boldsymbol {\dot {\oplus}}}  \q \xrightarrow{p_A} A \rightarrow 0\,.  \quad \quad $$
\end{itemize}
\end{lemma}
\vskip -28 pt \hfill $\Box$ \noindent \vskip 15pt We set
$$\da:=f^{\Join} (A {\boldsymbol {\dot {\oplus}}}  \q),\quad  \Gamma(f)  :=\{(a,f(a)) \mid  a\in A\}.$$
Clearly,   $\Gamma(f)\subseteq\da$ and they are subrings of $ A\times B$.  The motivation for replacing $A
{\boldsymbol {\dot {\oplus}}}  \q$ with its canonical image $\da$ inside $A\times B$ (under $f^{\Join} $) is
related to the fact that the multiplicative structure defined in $A {\boldsymbol {\dot {\oplus}}}  \q$, which
looks somewhat ``artificial'', becomes the restriction to $\da$ of the natural multiplication defined
componentwise  in the direct product $A\times B$. The ring $\da$ will be called {{\it the amalgamation of $A$
with $B$ along $\q$, with respect to $f:A\rightarrow B$}}.

\begin{example} \label{duplic}
 {\it The amalgamated duplication of a ring.}\\
A particular case of the construction  introduced above is the amalgamated duplication of a ring \cite{d'a-f-1}.
Let $A$ be a commutative ring with unity, and  let $E$ be an $A$--submodule of the total ring of fractions
$\tot(A)$ of $A$ such that $E\cdot E\subseteq E$. In this case, $E$ is an ideal in the subring $B:=(E:E) \
(:=\{z\in \tot(A) \mid zE\subseteq E\})$ of $\tot(A)$. If $ \iota:  A\rightarrow B$ is the  natural embedding,
then $A\!\Join^{\iota}\!\!E$ coincides with  $A\!\Join\!E$,  the amalgamated duplication of $A$ along $E$, as
defined in \cite{d'a-f-1}. A particular and relevant case is when $E :=I$ is an ideal in $A$. In this case, we can
take $B:=A$, we can consider the identity map $\ude := \ude_{A} : A \rightarrow A$ and we have  that $A\!\Join \!
I$, the amalgamated duplication of $A$ along the ideal $I$, coincides with  $A\!\Join^{\udes}\! \! I$, that we
will call also \it the simple amalgamation of $A$ along $I$  \rm(instead of { the amalgamation of $A$ along
$I$, with respect to $\ude_{A}$}).
\end{example}

\begin{example}\label{significativo}
 {\it The constructions  $A+{\boldsymbol X}B[{\boldsymbol X}]$ and $A+{\boldsymbol X}B[\![{\boldsymbol X}]\!]$.}\\
Let $A \subset B$  be an extension of commutative rings and ${\boldsymbol X}:=\{X_1,X_2, \z,X_n\}$ a finite set of
indeterminates over $B$.   In the polynomial ring $B[{\bf X}]$,  we can consider the following subring
$$A+{\boldsymbol X}B[{\boldsymbol X}]:=\{h \in B[{\boldsymbol X}] \mid h({\bf 0})\in A\}\,,$$ where ${\bf 0}$ is
the $n-$tuple whose components are $0$. This is a particular case of the general construction introduced above. In
fact, if $\sigma':A\hookrightarrow B[{\boldsymbol X}]$ is the natural embedding and $\q': ={\boldsymbol
X}B[{\boldsymbol X}]$, then it is easy to check that $A\!\Join^{\sigma'}{\!\!}\q'$ is isomorphic to
$A+{\boldsymbol X}B[{\boldsymbol X}]$
 (see also  the following    Proposition \ref{inizio}(3)).

Similarly, the subring $A+{\boldsymbol X}B[\![{\boldsymbol X}]\!]: =\{h\in B[\![{\boldsymbol X}]\!] \mid h({\bf
0})\in A\}$ of the ring of power series  $B[\![{\boldsymbol X}]\!]$ is isomorphic to
$A\!\Join^{\sigma''}{\!\!}\q''$, where $\sigma'': A\hookrightarrow B[\![{\boldsymbol X}]\!]$ is the  natural
embedding  and $\q'':={\boldsymbol X}B[\![{\boldsymbol X}]\!]$.
\end{example}

\begin{example}\label{D+M} {\it The $D+M$ construction.}\\
Let $M$ be a maximal ideal of a ring (usually, an integral domain) $T$ and let $D$ be a subring of $T$ such that
$M \cap D = (0)$. The ring $D+M := \{x +m \mid x \in D,\ m \in M \}$ is canonically isomorphic to
$D\!\Join^\iota\!\!M$, where $\iota: D \hookrightarrow T$ is the natural embedding.

More generally, let $\{M_\lambda \mid \lambda \in \Lambda\}$ be  a subset of the set of the maximal ideals of $T$,
such that $M_\lambda \cap D = (0)$ for all $\lambda \in \Lambda$, and set $J:= \bigcap_{\lambda \in \Lambda}
M_\lambda$. The ring $D+ J := \{x +j \mid x \in D,\  j \in J \}$ is canonically isomorphic to
$D\!\Join^\iota\!\!J$. In particular, if $D:= K$ is a field contained in $T$ and $J := \Jac(T)$ is the Jacobson
ideal of (the $K$--algebra) $T$, then $K+\Jac(T)$ is canonically isomorphic to   $K\!\Join^\iota\!\!\Jac(T)$,
where $\iota: K \hookrightarrow T$ is the natural embedding.
\end{example}

\begin{example}{\it  The  CPI--extensions  (in the sense of Boisen-Sheldon \cite{bo}).  }\\
Let $A$ be a ring and $P$ be a prime ideal of $A$.  Let $\boldsymbol{k}(P)$ be the residue field of the
localization $A_P$ and denote by  $\psi_P$ (or simply, by $\psi$)  the canonical surjective ring homomorphism
$A_P\longrightarrow \boldsymbol{k}(P)$.   It is wellknown that $\boldsymbol{k}(P)$ is canonically isomorphic to
the  quotient field of $A/P$, so we can identify $A/P$ with its canonical image into $\boldsymbol{k}(P)$.
 Then the subring $ \boldsymbol{C}(A,P)  :=\psi^{-1}(A/P)$ of $A_P$ is called the
 {\it CPI--extension of $A$ with respect to $P$}.
It is immediately seen that, if  we denote by $\lambda_P$ (or, simply, by $\lambda$) the localization homomorphism
$A\longrightarrow A_P$,  then $ \boldsymbol{C}(A,P)  $  coincides with  the ring $\lambda(A)+PA_P$.   On the other
hand,  if $J:=PA_P$,  we can consider $A\Join^{\lambda}J$ and we have the canonical projection  $A\Join^{\lambda}J
\rightarrow \lambda(A)+PA_P$,  defined by $(a, \lambda(a) + j) \mapsto   \lambda(a) + j $, where $a \in A$ and $j
\in PA_P$. It follows that  $ \boldsymbol{C}(A,P)$ is canonically isomorphic to $(A\Join^{\lambda}PA_P)/(P\times
\{0\})$ (Proposition \ref{inizio}(3)).

 More generally, let $I$ be an ideal of $A$ and let  $S_I$   be the set of
the elements $s\in A$ such that $s+I$ is a regular element of $A/I$. Obviously,   $S_I$  is a multiplicative
subset of $A$   and if  $\overline{S_I}$  is its canonical projection onto $A/I$,  then $\tot(A/I) =
(\overline{S_I})^{-1}(A/I) $.  Let  $\varphi_I: S^{-1}A\longrightarrow \tot(A/I)$  be the canonical surjective
ring homomorphism defined by $\varphi_I({a}{s}^{-1}):=(a+I)(s+I)^{-1}$, for all ${a} \in A$ and $s \in S$.   Then,
the subring $ \boldsymbol{C}(A,I):=  \varphi_I^{-1}(A/I) $ of  $S_I^{-1}A$   is called the {\it CPI--extension of
$A$ with respect to $I$}.  If  $\lambda_I: A\longrightarrow S_I^{-1}A$   is the localization homomorphism, then it
is easy to see that $ \boldsymbol{C}(A,I)$ coincides with the ring  $\lambda_I(A)+S_I^{-1}I$.  It will follow by
Proposition \ref{inizio}(3) that,   if we consider the ideal $J:=S_I^{-1}I$ of $S_I^{-1}A$, then $
\boldsymbol{C}(A,I)$ is  canonically  isomorphic to  $(A\Join^{\lambda_I}J)/(\lambda_I^{-1}(J)\times \{0\})$.
\end{example}

\begin{remark}{\it Nagata's idealization.}\\
Let $A$ be a commutative ring and $\mathcal M $ a $A$--module. We recall that, in 1955,
  Nagata introduced the ring extension of $A$ called
 \it  the idealization of $\mathcal M $ in $A$,  \rm
denoted here by $A\ltimes\! \mathcal M $, as the $A$--module
$A\oplus \mathcal M $ endowed with a multiplicative structure defined by:
$$
(a,x)(a',x') := (aa',ax'+a'x)\,, \; \mbox{ for all    $a, a' \in A$
and $x,x' \in \mathcal M$ }
$$
\noindent  (cf. \cite{na}, Nagata's book \cite[page 2]{n}, and Huckaba's book \cite[Chapter VI, Section 25]{h}).
The idealization $A\!\ltimes\! \mathcal M $   is a ring, such that  the canonical embedding $ \iota_A :   A
\hookrightarrow A\!\ltimes\! \mathcal M$ (defined by $a \mapsto (a, 0)$,   for all $a \in A$)   induces  a subring
$A^\ltimes   \ (:= \iota_A(A))$      of $A\! \ltimes \! \mathcal M $  isomorphic to $A$ and the embedding
$\iota_{\mathcal M }: \mathcal M  \hookrightarrow A\!\ltimes\! \mathcal M$  (defined by $x \mapsto (0, x)$,  for
all $x \in \mathcal M $) determines an ideal  $\mathcal M ^\ltimes   \ (:= \iota_{\mathcal M }(\mathcal M ))$   in
$A\!\ltimes\! \mathcal M $ (isomorphic, as an $A$--module, to $\mathcal M $), which is nilpotent of index $2$
(i.e.  ${\mathcal M}^\ltimes\!\cdot\! {\mathcal M }^\ltimes = 0$).

 For the sake of simplicity, we will identify
$\mathcal M $ with ${\mathcal M }^\ltimes$ and $A$ with $A^\ltimes$.
If $p_A: A\!\ltimes\! \mathcal M  \rightarrow A$ is the canonical projection (defined by $(a, x) \mapsto a$,
for all $a \in A$ and $x \in \mathcal M$),  then
$$
0 \rightarrow \mathcal M  \xrightarrow {\iota_{_{\!\mathcal M }} }  A\!\ltimes\! \mathcal M  \xrightarrow {p_{_{\!A}}} A \rightarrow 0
$$
is a spitting exact sequence of $A$--modules.
(Note that the idealization $A\!\ltimes\! \mathcal M $ is also called  in \cite{fs}   {\it {the trivial
extension of $A$ by $\mathcal M $}}.)

We can apply the construction of Lemma \ref{le:1}  by taking $ \mathcal{R} :=\mathcal M $,   where $\mathcal M$ is
{an} $A$--module,  and considering $\mathcal M $ as a (commutative) ring without identity, endowed with a
trivial multiplication (defined by $x\!\cdot\! y := 0$  for all $x, y \in \mathcal M $).  In this way,  we have
that the Nagata's idealization is a particular case of the construction considered in Lemma \ref{le:1}. Therefore,
the Nagata's idealization can be interpreted as a particular case of the general  amalgamation construction.  Let
$B:= A\!\ltimes\! \mathcal M $ and $ \iota \ (=\iota_A)  : A \hookrightarrow B$ be the canonical embedding.  After
identifying  $\mathcal M $ with ${\mathcal M}^\ltimes$, $\mathcal M $ becomes an ideal of $B$. It is now
straighforward that $A\!\ltimes\! \mathcal M $ coincides with the amalgamation $A\!\Join^{ \iota  }\!\!\mathcal M
$.

Although this, the Nagata's idealization and the constructions of the type $\da$ can be  very different from an
algebraic point of view. In fact, for example, if $\mathcal M $ is a nonzero $A$--module, the ring $A\!\ltimes\!
\mathcal M $ is  always  not reduced (the element $(0, x)$ is nilpotent, for all $x\in \mathcal M $), but the
amalgamation $\da$ can be an integral domain (see  Example \ref{D+M} and Proposition \ref{dom}).
\end{remark}

 \section{Iteration of the construction $\da$}

{\sl  In the following all rings will always be commutative with identity, and every ring homomorphism will send 1
to 1}.


\smallskip

If $A$ is a ring and $I$ is an ideal of $A$, we can consider the amalgamated duplication of the ring $A$ along its ideal $I$ (=  the simple amalgamation of $A$ along $I$),
 i.e.,  $A\!\Join\! I: =\{(a, a+i) \mid a\in A ,\ i\in I\}$  (Example \ref{duplic}).   For the sake of simplicity, set $A' := A\!\Join\! I$.
It is immediately seen that $I':=\{0\}\!\times\! I$ is an ideal of $A'$, and thus we can consider again the simple amalgamation of $A'$   along  $ I'$, i.e., the ring  $A'':=A'\!\Join\! I'  \ (= (A\!\Join\!  I)\!\Join\! (\{0\}\!\times\! I)$).
It is easy to check that the ring $A''$ may not be considered as a simple amalgamation of $A$   along  one of its ideals.   However, we can show that $A''$ can be interpreted as an amalgamation of algebras, giving in this way an answer to a problem posed by B. Olberding in 2006 at Padova's Conference in honour of L. Salce.

\medskip

We start by showing that it is possible to  iterate the amalgamation of algebras  and the result is still    an
amalgamation of algebras.

More precisely, let $f: A\rightarrow B$ be a ring   homomorphism   and $\q$ an ideal of $B$. Since
$\q^{\prime_f}:=\{0\}\times \q$ is an ideal of the ring $A^{\prime_f}:=\da$, we can consider the simple
amalgamation of $A^{\prime_f}$ along $J^{\prime_f}$,
  i.e., $A^{\prime\prime_f} :=A^{\prime_f}\!\Join\! J^{\prime_f}$ (which coincides with  $A^{\prime_f}\!\Join^{\udes}\!\! \q^{\prime_f}$,
  where $\ude$
  $ := {\ude}_{A^{\prime_f}}$
   is the identity mapping of $A^{\prime_f}$~). On the other hand, we can consider the mapping $f^{(2)}: A\rightarrow B^{(2)}  := B \times B$,  defined by $a\mapsto (f(a),f(a))$ for all $a \in A$. Since $\q^{(2)}  :=\q\times\q$  is an ideal  of the ring  $B^{(2)}$, we can consider the amalgamation $A\!\Join^{f^{(2)}}\!\!\q^{(2)}$.  Then,  the mapping
$
A^{\prime\prime_f} \rightarrow A\Join^{f^{(2)}}\q^{(2)},
$ defined by
$
((a,f(a)+j_1),(a,f(a)+j_1)+(0, j_2))\mapsto (a, (f(a),f(a))+(j_1, j_1+j_2))
$ for all $a \in A$ and $j_1, j_2 \in J$,
is a ring isomorphism,  having as inverse map the map $
 A\! \Join^{f^{(2)}}\!\!\!\q^{(2)}\rightarrow A^{\prime\prime_f},
$
defined by $ (a, (f(a)+j_1,  f(a)+j_2)) \mapsto ((a ,f(a)+j_1),  (a ,f(a)+j_1) +(0, j_2 - j_1)) $ for all $a \in
A$ and $j_1, j_2 \in J$.
We will denote by $A\!\Join^{2,f}\!\! J$ or, simply, $A^{(2, f)}$ (if no confusion can
arise)  the ring $A\!\Join^{f^{(2)}}\!\! \q^{(2)}$, that we will call the {\it 2-amalgamation of the $A$--algebra
$B$ along $\q$ (with respect to $f$)}.

For $n\geq 2$, we   define the {\it n-amalgamation of the $A-$algebra $B$ along $\q$ (with respect to $f$)} by
setting
$$
A\!\Join^{n,f}\!\!\q:=A^{(n, f)}:=A\! \Join^{f^{(n)}}\!\! J^{(n)},
$$
where
 $f^{(n)}: A\rightarrow B^{(n)}:= B \times B \times ... \times B \mbox{  ($n$--times)}$
 is the diagonal homomorphism associated to $f$
 and $J^{(n)} := J \times J \times ... \times J \mbox{  ($n$--times)}$.  Therefore,
$$
 A\!\Join^{n,f}\!\!\q =\{ (a, (f(a), f(a), ... , f(a)) +(j_1, j_2, ... , j_n)) \mid a \in A,  \ j_1, j_2, ... , j_n \in J \}\,.
 $$

 \begin{proposition}\label{iter} Let $f: A\rightarrow B$ be a ring homomorphism  and $\q$ an ideal of $B$.
 Then $A\!\Join^{n,f}\!\! \q$
 is canonically isomorphic to the simple amalgamation
 $A^{(n -1, f)}\!\Join \!\! J^{(n -1, f)} \  (=  A^{(n -1, f)}\Join^{\udes}\!\! J^{(n -1, f)})$,
 where  $J^{(n -1, f)}$ is the canonical isomorphic image of $J$ inside  $A^{(n -1, f)}$  and \,
 $\ude := {\ude}_{A^{(n -1, f)}}$
is the identity mapping of $A^{(n -1, f)}$.
\end{proposition}
\begin{proof} The proof can be given by  induction on $n \geq 2$.
For the sake of simplicity, we only consider here the inductive step from $n=2$ to $n+1 \ (=3)$. It is
straightforward that the mapping $A\!\Join^{3,f}\!\! \q \rightarrow A^{\prime\prime_f}\!\!\Join\!\!
J^{\prime\prime_f}$, defined by $ (a, (f(a), f(a), f(a)) +(j_1, j_2, j_3))$ $ \mapsto$ $(a'', a'' +j'') $, where
$a'' := ((a, f(a)+j_1), (a, f(a)+j_1)+(0, j_2-j_1)) \in  A^{\prime\prime_f}$ and $j'' :=  ((0,0), (0, j_3 -j_2))
\in  J^{\prime\prime_f}$, for all $a \in A$ and $ j_1, j_2, j_3 \in J$ establishes a canonical ring isomorphism.
\end{proof}

 In particular, let
  $A$ be a ring and $I$ an ideal of $A$, the simple amalgamation of $A' := A\!\Join\! I$ along $I' := \{0\} \times I$,
  that is $A'' := A' \!\! \Join \!\! I'$,  is canonically isomorphic to  the 2-amalgamation
  $A\!\Join^{2,\udes}\!\! I = \{ (a, (a, a) +(i_1, i_2)) \mid a\in A, \ i_1, i_2 \in I \}$.

\begin{example}
We can apply the previous (iterated) construction to curve singularities. Let $A$ be the ring of an algebroid
curve with $h$ branches (i.e.,  $A$ is a one-dimensional reduced ring of the form $K[\![X_1, X_2,
\z,X_r]\!]/\bigcap_{i=1}^hP_i$, where $K$ is an algebraically closed field, $X_1,X_2, \z,X_r$ are indeterminates
over $K$ and $P_i$ is an height $r-1$ prime ideal of $K[\![X_1,X_2, \z,X_r]\!]$, for $1 \leq i \leq r$). If $I$ is
a regular and proper ideal of $A$, then, with an argument similar to that used in the proof of
 \cite[Theorem 14]{d'a}  (where the case of a simple amalgamation of the ring of the given
 algebroid curve  is investigated), it can be shown  that $A\Join^nI$ is the ring of an algebroid curve with $(n+1)h$
 branches; {moreover, for each branch of $A$, there are exactly $n+1$ branches of $A\Join^nI$ isomorphic to
 it.}
\end{example}

\section{Pullback constructions}

Let $f:A\longrightarrow B$   be a ring homomorphism and $\q$ an ideal of $B$. In the remaining part of the paper,
we intend to investigate the algebraic properties of the ring $\da$, in relation with those of $A,B,J$ and $f$.
One important tool we can use for this purpose is the fact that the ring $\da$ can be represented as a  {pullback}
(see next Proposition \ref{pull}). On the other hand, we will provide a characterization of  those pullbacks that
give rise to amalgamated algebras (see next Proposition \ref{fibret}). After proving these facts, we will make
some pertinent remarks useful for the subsequent investigation {on
amalgamated algebras.}

\begin{definition}\label{de} \rm
We recall that, if $\alpha:A\rightarrow C,\,\,\, \beta:B\rightarrow C$ are ring homomorphisms, the subring
$D:=\alpha\times_{_C}\beta:=\{(a,b)\in A\times B \mid \alpha(a)=\beta(b)\}$ of $A\times B$ is called the
\textit{{pullback}} (or \textit{fiber product}) of $\alpha$ and $\beta$.
\end{definition}
The fact tat $D$ is a  {pullback} can also be described by saying that the triplet $(D, p_A, p_B)$ is a solution
of the universal problem of rendering commutative the diagram built on $\alpha$ and $\beta$
$$
\begin{CD}
D  @> p_{_A}>> A\\
@V p_{_B}VV @ V\alpha VV\\
B @> \beta >> C
\end{CD}
$$
 where   $p_{_A}$ (respectively, $p_{_B}$) is  the restriction to $\alpha\times_{_C}\beta$ of the projection of $A\times B$  onto $A$ (respectively, $B$).

\begin{proposition}\label{pull}
 Let $f:A\rightarrow B$ be a ring homomorphism and $\q$ be an ideal of $B$. If $\pi:B\rightarrow B/J$ is the canonical projection and $\breve f:=\pi\circ f$, then $\da=\breve f\times_{_{B/J}}\pi$.
\end{proposition}
\begin{proof} The statement follows easily from the definitions.\end{proof}

\begin{remark}\label{nuova}
Notice that we have many other ways to describe the ring $\da$ as a  pullback. In fact, if $C:=A\times B/J$ and
$u:A\rightarrow C,\, v:A\times B\rightarrow C$ are the canonical ring homomorphisms defined by
$u(a):=(a,f(a)+J),\, v((a,b)):=(a,b+J)$, for every $(a,b)\in A\times B$, it is straightforward to show that $\da$
is canonically isomorphic to $u\times_{_C}v$. On the other hand, if $\aaa:=f^{-1}(\q)$,\
  $  {\breve{u}} :A/\aaa\rightarrow A/\aaa\times B/\q$ and ${\breve{v}}:A \times B\rightarrow A/\aaa\times B/\q$
  are the natural ring homomorphisms induced by $u$ and $v$, respectively, then $\da$ is also
   canonically isomorphic to the {pullback} of ${\breve{u}}$ and ${\breve{v}}$.
\end{remark}

The next goal is to show that the rings of the form $\da$,
for some ring homomorphism $f:A\rightarrow B$ and some ideal $J$ of $B$,
 determine a distinguished subclass of the class of all fiber products.

\begin{proposition}\label{factor}
Let $A,B,C,\alpha,\beta$ as in Definition \ref{de}, and let $f:A\rightarrow B$ a ring homomorphism. Then the following conditions are equivalent.
\begin{enumerate}
\item[\rm(i)] There exist an ideal $J$ of $B$ such that $\da$ is the fiber product of $\alpha$ and $\beta$.
\item[\rm(ii)] $\alpha$ is the composition $\beta\circ f$.
\end{enumerate}
If the previous conditions hold, then $J=\Ker(\beta)$.
\end{proposition}
\noindent

\begin{proof} Assume condition (i) holds, and let $a$ be an element of $A$.
Then $(a,f(a))\in \da$ and, by assumption, we have $\alpha(a)=\beta(f(a))$. This prove condition  (ii).

 Conversely, assume that $\alpha=\beta\circ f$.
 We want to show that   the ring $A\Join^f\Ker(\beta)$ is the fiber product of $\alpha$ and $\beta$.
 The inclusion $A\Join^f\Ker(\beta)\subseteq \alpha\times_{_C}\beta$ is clear.
 On the other hand, let $(a,b)\in \alpha\times_{_C}\beta$.
 By assumption, we have $\beta(b)=\alpha(a)=\beta(f(a))$.
This shows that $b-f(a)\in\Ker(\beta)$, and thus $(a,b)=(a,f(a)+k)$, for some $k\in\Ker(\beta)$. Then
$A\Join^f\Ker(\beta)= \alpha\times_{_C}\beta$ and condition (i) is true.

The last statement of the proposition is straightforward.  \end{proof}

In the previous proposition we assume the existence of the ring homomorphism $f$. The next step is to give a
condition for the existence of $f$. We start by recalling that a ring homomorphism $r:B\rightarrow A$ is called
{\it a ring retraction} if there exists a ring homomorphism  $\iota:A\rightarrow B$, such that $r\circ \iota={\rm
id}_{A}$. In this situation,  $\iota$ is necessarily injective, $r$ is necessarily surjective, and
 $A$ is called a {\it retract of $B$}.

\begin{example}\label{retraction}
 If  $r:B\rightarrow A$ is a ring retraction and $\iota:A\hookrightarrow B$ is  a ring embedding such that
$r\circ\iota={\rm{id}}_A$, then $B$ is naturally isomorphic to $A\Join^{\iota}\!\Ker(r)$.  This is a consequence
of the facts, easy to verify, that $B=\iota(A)+\Ker(r)$ and that $\iota^{-1}(\Ker(r))=\{0\}$ (for more details see
next Proposition \ref{inizio}(3)).
\end{example}

\begin{remark}\label{farfret}
Let $f:A\rightarrow B$ be a ring homomorphism and $J$ be an ideal of $B$. Then $A$ is a retract of $\da$. More
precisely, $\pi_{_A}:\da\rightarrow A$, $(a,f(a),j)\mapsto a$, is a retraction, since the map $\iota:A\rightarrow
\da$, $a\mapsto (a,f(a))$, is a ring embedding  such that $\pi_{_A}\circ\iota={\rm id}_A$.
\end{remark}

\begin{proposition}\label{fibret}
Let $A,B,C,\alpha,\beta,p_{_A},p_{_B}$ be as in Definition \ref{de}. Then,  the following conditions are equivalent.
\begin{enumerate}
\item[\rm(i)] 
  $p_{_A}:\alpha\times_{_C}\beta\rightarrow A$ is a ring retraction.
\item[\rm(ii)] There exist an ideal $J$ of $B$ and a ring homomorphism $f:A\rightarrow B$ such that $\alpha\times_{_C}\beta=\da$.
\end{enumerate}
\end{proposition}
\begin{proof}  Set $D:=\alpha\times_{_C}\beta$. Assume that condition (i) holds and
let $\iota:A\hookrightarrow D$ be a ring embedding such that $p_{_A}\circ\iota={\rm id}_A$. If we consider the
ring homomorphism $f:=p_{_B}\circ\iota:A\rightarrow B$, then, by using the definition of a pullback, we have
$\beta\circ f=\beta\circ p_{_B}\circ \iota=\alpha\circ p_{_A}\circ \iota=\alpha\circ {\rm id}_A=\alpha$. Then,
condition (ii) follows by applying Proposition \ref{factor}.   Conversely, let $f:A\rightarrow B$ be a ring
homomorphism such that $D=\da$, for some ideal $J$ of $B$. By Remark \ref{farfret}, the projection of $\da$ onto
$A$ is a ring retraction.\end{proof}

\begin{remark} Let $f, \ g :A\rightarrow B$ be two ring homomorphisms and $J$ be an ideal of $B$.
It can happen that $\da=A\!\Join^g\!\!J$,  with $f \neq g$. In fact,
 it is easily seen that $\da=A\!\Join^g\!\!J$ if and only if $f(a)-g(a)\in J$, for every $a\in A$. \\
For example, let $f,g:A[X]\rightarrow A[X]$ be the ring  homomorphisms defined by   $f(X):=X^2,\ f(a) := a, \
g(X):=X^3, \ g(a) := a$, for all $a \in A$,  and set $J:=XA[X]$. Then $f\neq g$, but $A[X]\Join^f\!J=A[X]\Join^g\!
J$, since $f(p)-g(p)\in J$, for all  $p\in A[X]$.
\end{remark}

The next goal  is to give some sufficient conditions for a pullback to be reduced. Given a ring $A$, we denote by
${\rm Nilp}(A)$ the ideal of all nilpotent elements of $A$.

\begin{proposition}\label{prid}
 With the notation of Definition \ref{de}, we have:
\begin{enumerate}
 \item [\rm(1)] If $D \ (=\alpha\times_{_C}\beta)$ is reduced, then

{\ } ${\rm Nilp}(A)\cap \Ker(\alpha)=\{0\}$ \ and \ ${\rm Nilp}(B)\cap \Ker(\beta)=\{0\}.$
\item[\rm(2)] If at least one of the following conditions holds
\begin{enumerate}
 \item [\rm(a)] $A$ is reduced and  ${\rm Nilp}(B)\cap \Ker(\beta)=\{0\}$,
 \item[\rm(b)] $B$ is reduced and ${\rm Nilp}(A)\cap \Ker(\alpha)=\{0\}$,
\end{enumerate}
then $D$ is reduced.
\end{enumerate}
\end{proposition}
\begin{proof}  (1) Assume $D$ reduced. By simmetry, it sufficies to show that
${\rm Nilp}(A)\cap \Ker(\alpha)=\{0\}$. If $a\in{\rm Nilp}(A)\cap \Ker(\alpha)$, then $(a,0)$ is a nilpotent
element of $D$, and thus $a=0$.

(2) By the simmetry of conditions (a) and (b), it is enough to show that, if condition (a) holds, then $D$ is
reduced. Let $(a,b)$ be a nilpotent element of $D$. Then $a=0$, since $a\in {\rm Nilp}(A)$ and $A$ is reduced.
Thus we have $(a,b)=(0,b)\in{\rm Nilp}(D)$, hence $b\in{\rm Nilp}(B)\cap\Ker(\beta)=\{0\}$.
\end{proof}
 We study next the Noetherianity of a ring arising from a  {pullback} construction as in Definition \ref{de}.

\begin{proposition}\label{noepull}
 With the notation of Definition \ref{de}, the following conditions are equivalent.
\begin{enumerate}
 \item [\rm(i)] $D \ (=\alpha\times_{_C}\beta)$ is a Noetherian ring.
 \item [\rm(ii)] $\Ker(\beta)$ is a Noetherian $D$--module (with the  $D$--module   structure naturally induced by $p_{_B}$) and $p_{_A}(D)$ is a Noetherian ring.
\end{enumerate}
\end{proposition}
\begin{proof} It is easy to see that  $\Ker(p_{_A})=\{0\}\times \Ker(\beta)$. Thus,  we have the following short exact sequence of $D$--modules
$$0\longrightarrow\Ker(\beta)\stackrel{i }{\longrightarrow}D\stackrel{p_{_A}}{\longrightarrow}p_{_A}(D)
\longrightarrow 0,$$
where $ i$ is the natural $D$--module embedding (defined by $x\mapsto (0,x)$ for all $ x \in \Ker(\beta)$). By  \cite[Proposition (6.3)]{am}, $D$ is a Noetherian ring if and only if $\Ker(\beta)$ and $p_{_A}(D)$ are Noetherian as $D$--modules. The statement now follows immediately, since the $D$--submodules of $p_{_A}(D)$ are exactly the ideals of the ring $p_{_A}(D)$. \end{proof}

\begin{remark}  Note  that, in Proposition \ref{noepull},
we did not require $\beta$ to be surjective. However, if $\beta$ is surjective, then $p_A$ is also surjective and
so $p_A(D) = A$. Therefore, in this case, $D$ is a Noetherian ring if and only if $A$ is a Noetherian ring and
$\Ker(\beta)$ is a Noetherian $D$--module.
\end{remark}

\section{The ring $\da$: some  basic algebraic properties}

We start with some straightforward  consequences of the definition of amalgamated algebra   $\da$.

\begin{proposition}
\label{inizio}   Let $f: A\rightarrow B$ be a ring   homomorphism,  $\q$ an ideal of $B$  and let $\da := \{ (a,
f(a)+j) \mid a\in A, \ j \in J \}$ be as in Section 2.
\begin{enumerate}
 \item[\rm (1)]  Let  $\iota := \iota_{A, f, J}: A\rightarrow \da$ be the natural
the ring homomorphism defined by $\iota(a)  := (a, f(a))$, for all $a \in A$. Then $\iota$ is an embedding, making
$\da$   a  ring   extension of $A$ \  (with $\iota(A) =  \Gamma(f) \  (:=\{(a,f(a)) \mid a\in A\}$  subring of
$\da$).

\item[\rm (2)]  Let $I$ be an ideal of $A$ and set $ I\!\Join^f\!\! J :=\{(i, f(i)+j) \mid i\in I, j \in J \}$.
Then $I\!\Join^f\!\! J$  is an ideal of $\da$, the composition  of canonical homomorphisms
$A\stackrel{\iota}{\hookrightarrow} \da\twoheadrightarrow \da/I\!\Join^f\!\! J$ is a surjective ring homomorphism
and its kernel coincides with  $I$.\\ Hence, we have the following canonical isomorphism:
 $$
\frac{\da}{I\!\Join^f\!\! J}\cong \frac{A}{I}\,.
$$

 \item[\rm (3)] Let $p_{_A}: \da \rightarrow A$ and $p_{_B}:\da\rightarrow B$ be the natural projections of
 $\da \subseteq A \times B$ {into} $A$ and $B$, respectively.  Then $p_{_A}$ is surjective
 and\, $\Ker(p_{_A})=\{0\}\times J$.\\
Moreover, $p_{_B}(\da)=f(A)+ J$ and $\Ker(p_{_B})=f^{-1}(\q)\times \{0\}$.  Hence, the following canonical
isomorphisms  hold:
$$
\frac{\da}{(\{0\}\times \q)}\cong A\quad \mbox{ and } \quad\frac{\da}{f^{-1}(\q)\times \{0\}}\cong f(A)+\q \,.
$$
\item[\rm(4)] Let $\gamma:\da\rightarrow (f(A)+J)/J$
be the natural ring homomorphism, defined by $(a,f(a)+j)\mapsto f(a)+J$. Then $\gamma$ is surjective and
$\Ker(\gamma)=f^{-1}(J)\times J$. Thus, there exists a natural isomorphism
$$
\frac{\da}{f^{-1}(J)\times J}\cong\frac{f(A)+J}{J}\,.
$$
In particular, when $f$ is surjective we have
$$
\frac{\da}{f^{-1}(J)\times J}\cong \frac{B}{J}\,.
$$
 \end{enumerate}
\end{proposition}
\vskip -34pt \hfill $\Box$
 \vskip 25pt

The ring $B_\diamond :=f(A)+\q$ (which is a subring of $B$) has an important role in the construction $\da$.\  For
instance,  if $f^{-1}(\q)=\{0\}$,  we have $\da\cong B_\diamond$ (Proposition \ref{inizio}(3)).     Moreover, in
general,  $J$ is an ideal also in $B_{\diamond}$ and, if we denote by $f_{\diamond}: A \rightarrow B_{\diamond}$
the ring homomorphism induced from $f$, then $A\!\Join^{f_{\diamond}}\!\! J = A \!\Join^f\!\! J$.  The next result
shows one more aspect of the essential role of the ring $B_{\diamond}$ for the construction $\da$.

\medskip
\begin{proposition}\label{dom}
  With the notation of Proposition \ref{inizio}, assume $\q\neq\{0\}$. Then,   the following conditions are equivalent.
\begin{enumerate}
 \item[\rm(i)] $\da$ is an integral domain.
 \item[\rm(ii)] $f(A)+\q$ is an integral domain and $f^{-1}(\q)=\{0\}$.
\end{enumerate}
In particular, if $B$ is an integral domain and $f^{-1}(\q)=\{0\}$, then $\da$ is an integral domain.
\end{proposition}
\begin{proof}  (ii)$\Rightarrow$(i) is obvious, since  $f^{-1}(\q)=\{0\}$ implies that $\da\cong f(A)+\q$ (Proposition \ref{inizio}(3)).

 Assume that condition (i) holds.
 If there exists an element $a\in A\w\{0\}$ such that $f(a)\in\q$, then $(a,0)\in (\da)\w\{(0,0)\}$.
 Hence, if $j$ is a nonzero element of $\q$, we have $(a,0)(0,j)=(0,0)$, a contradiction.
 Thus $f^{-1}(\q)=\{0\}$. In this case, as observed above,  $\da\cong f(A)+\q$ (Proposition \ref{inizio}(3)), so  $f(A)+\q$ is an integral domain.\end{proof}

\begin{remark}
(1) \rm Note that, if $\da$ is an integral domain, then   $A$ is also an integral domain,  by Proposition
\ref{inizio}(1).

(2) \rm Let $B =A$,  $f = \ude_A$ and $J=I$ be an ideal of $A$.  In this situation,  $A\!\Join^{{\udes}_{_{\!
A}}}\!\! I$ (the simple amalgamation of $A$ along $I$) coincides with the amalgamated duplication of $A$ along $I$
(Example \ref{duplic}) and it is never an integral domain, unless $I =\{0\}$ and $A$ is an integral domain.
\end{remark}
Now, we characterize when the amalgamated algebra  $\da$ is a reduced ring.

\begin{proposition}\label{farfridotta}
We preserve the notation of Proposition \ref{inizio}. The following conditions are equivalent.
\begin{enumerate}
 \item [\rm(i)] $\da$ is a reduced ring.
 \item [\rm(ii)] $A$ is a reduced ring and ${\rm Nilp}(B)\cap\q=\{0\}$.
\end{enumerate}
In particular,  if $A$ and $B$ are reduced, then $\da$ is reduced;  conversely, if $\q$ is a radical ideal of $B$ and $\da$ is reduced,  then  $B$  (and $A$) is reduced.
\end{proposition}
\begin{proof}  From Proposition \ref{prid}{(2, a)} we deduce easily that
(ii)$\Rightarrow$(i),  after noting that, with  the notation of Proposition \ref{pull},  in this case $
\Ker(\pi)=J. $

(i)$\Rightarrow$(ii) By Proposition \ref{prid}{(1)} and the previous equality,  it is enough to show that if $\da$
is reduced, then $A$ is reduced.  This is trivial   because, if $a\in{\rm Nilp}(A)$, then $(a, f(a))\in{\rm
Nilp}(\da)$.

Finally, the first part of the last statement is straightforward. As for the second part, we have $\{0\}={\rm
Nilp}(B)\cap J={\rm Nilp}(B)$ (since $J$ is radical, and so $J\supseteq {\rm Nilp}(B)$). Hence $B$ is
reduced.\end{proof}

\begin{remark} (1) Note that, from the previous result, when $B =A$, $f =\ude_A \ (=\ude)$
and $J=I$ is an ideal of $A$,  we reobtain easily that  $A\!\Join\!I \ (= A\!\Join^{\udes}\!\! I)$   is a reduced
ring if and only if $A$ is a reduced ring  \cite[Proposition 2.1]{d'a-f-2}.

   (2) The previous proposition implies that the property of being reduced for $\da$ is independent of the nature of $f$.

 (3) If $A$ and $f(A)+J$ are reduced rings, then $\da$ is a reduced ring, by Proposition \ref{farfridotta}. But
the converse is not true in general. As a matter of fact, let $A:=\mathbb Z$, $B:=\mathbb Z\times (\mathbb
Z/4\mathbb Z)$, $f:A\rightarrow B$ be the ring homomorphism such that $f(n)=(n,[n]_4)$, for every $n\in\mathbb
\inte$ (where $[n]_4$ denotes the class of $n$ modulo 4). If we set $J:=\mathbb Z\times \{[0]_4\}$, then $J\cap
{\rm Nilp}(B)=\{0\}$, and thus $\da$ is a reduced ring, but $(0,[2]_4)=(2,[2]_4)+(-2,[0]_4)$ is a nonzero
nilpotent element of $f(\mathbb Z)+J$.
\end{remark}

The next proposition provides an answer to the question  of when $\da$ is a Noetherian ring.
\begin{proposition}\label{noet}
  With the notation of Proposition \ref{inizio}, the following conditions are equivalent.
\begin{enumerate}
 \item[\rm(i)] $\da$ is a Noetherian ring.
 \item[\rm(ii)] $A$ and $f(A)+\q$ are Noetherian rings.
\end{enumerate}
\end{proposition}
\begin{proof} (ii)$\Rightarrow$(i). Recall that $\da$ is the fiber product of the ring homomorphism $\breve f:A\rightarrow B/J$ (defined  by $a\mapsto f(a)+J$) and of the canonical projection $\pi:B\rightarrow B/J$. Since the projection $p_{_A}:\da\rightarrow A$ is surjective (Proposition \ref{inizio}(3)) and $A$ is a Noetherian ring, by Proposition \ref{noepull}, it sufficies to show that $J(=\Ker(\pi))$, with the structure of $\da-$module induced by $p_{_B}$, is  Noetherian. But this fact is easy, since every $\da-$submodule of $J$ is an ideal of the Noetherian ring $f(A)+J$.

 (i)$\Rightarrow$(ii) is a straighforward consequence of Proposition \ref{inizio}(3).\end{proof}

Note that, from the previous result, when $B =A$,  $f =\ude_A \  (=\ude)$ and $J=I$ is an ideal of $A$,  we
reobtain easily that  $A\!\Join\!I \ (= A\!\Join^{\udes}\!\! I)$   is a Noetherian ring if and only if $A$ is a
Noetherian ring  \cite[Corollary 2.11]{d'a-f-1}.

However, the previous proposition    has a moderate interest because the Noetherianity of $\da$ is not directly related to the data (i.e., $A, B, f$ and $J$), but to the ring $B_{\diamond} = f(A) +J$ which is canonically isomorphic $\da$,
if $f^{-1}(\q)=\{0\}$ (Proposition \ref{inizio}(3)).  Therefore,  in order to obtain  more useful criteria for the Noetherianity of $\da$, we specialize Proposition \ref{noet} in some relevant cases.

\begin{proposition}\label{noepiubello}
  With the notation of Proposition \ref{inizio},  assume that at least one of the following conditions holds:
\begin{enumerate}
 \item[\rm(a)] $\q$ is a finitely generated $A$--module (with the structure naturally induced by $f$).
\item[\rm(b)] $\q$ is a Noetherian $A$--module (with the structure naturally induced by $f$).
 \item[\rm(c)] $f(A)+\q$ is Noetherian as $A$--module (with the structure naturally induced by $f$).
 \item[\rm(d)] $f$ is a finite homomorphism.
\end{enumerate}
Then $\da$ is Noetherian if and only if $A$ is Noetherian.\
 In particular, if $A$ is a Noetherian ring and $B$ is a Noetherian
 $A$--module (e.g., if $f$ is a finite homomorphism \cite[Proposition 6.5]{am}),
 then $\da$ is a Noetherian ring for all ideal $J$ of $B$.
\end{proposition}
\begin{proof} Clearly, without any extra assumption,   if $\da$ is a Noetherian ring,  then $A$ is a Noetherian ring, since it is isomorphic to $\da /(\{0\} \times J)$ (Proposition \ref{inizio}(3)).

Conversely, assume that  $A$ is a Noetherian ring. In this case, it is straighforward to verify that  conditions (a), (b), and (c) are equivalent \cite[Propositions 6.2, 6.3, and  6.5]{am}.   Moreover (d) implies (a), since $\q$ is an $A$--submodule of $B$, and $B$ is a Noetherian $A$--module under condition (d)  \cite[Proposition  6.5]{am}.

Using {the} previous observations, it is enough to show that $\da$ is Noetherian if $A$ is Noetherian and
condition (c) holds. If $f(A)+\q$ is Noetherian as an $A$--module, then $f(A)+\q$ is a Noetherian ring (every
ideal of $f(A)+\q$ is an  $A$--submodule of $f(A)+\q$). The conclusion follows from Proposition
\ref{noet}((ii)$\Rightarrow$(i)).

 The last statement is a consequence of the first part and of the fact that, if $B$ is a Noetherian $A$--module, then (a) holds \cite[Proposition 6.2]{am}.   \end{proof}

\begin{proposition}\label{finite}
We preserve the notation of Propositions \ref{inizio} and \ref{pull}. If $B$ is a Noetherian ring and the ring
homomorphism $\breve f:A\rightarrow B/J$ is finite, then $\da$ is a Noetherian ring if and only if $A$ is a
Noetherian ring.
\end{proposition}
\begin{proof} If $\da$ is Noetherian we already know that $A$ is Noetherian.
Hence, we only need to show that if $A$ and $B$ are Noetherian rings and $\breve f$ is finite then $\da$ is
Noetherian. But this fact follows immediately from
 \cite[Proposition 1.8]{F}.
\end{proof}
As a consequence of the previous proposition, we can characterize when rings of the form $A+XB[X]$ and
$A+XB[\![X]\!]$ are Noetherian. Note that S. Hizem and A. Benhissi \cite{hiz} have already given  a
characterization of the Noetherianity of the power series rings of the form $A+XB[\![X]\!]$.  The next corollary
provides a simple proof of Hizem and Benhissi's Theorem and shows that a similar characterization holds for the
polynomial case (in several indeterminates). At the Fez Conference in June 2008, S. Hizem has announced to have
proven a similar result in the polynomial ring case with a totally different approach.

\begin{corollary}\label{noegenerale}
Let $A \subseteq B$ be a ring extension and ${\boldsymbol X}:=\{X_1,\z,X_n\}$ a finite set of indeterminates over $B$. Then the following conditions are equivalent.
\begin{enumerate}
\item [\rm(i)] $A+{\boldsymbol X}B[{\boldsymbol X}]$ is a Noetherian ring.
\item [\rm(ii)] $A+{\boldsymbol X}B[\![{\boldsymbol X}]\!]$   is a Noetherian ring.
\item [\rm(iii)] $A$ is a Noetherian ring and $A \subseteq B$ is a finite ring extension.
\end{enumerate}
\end{corollary}
\begin{proof}  (iii)$\Rightarrow$(i, ii). With the notations of Example \ref{significativo}, recall that $A+{\boldsymbol X}B[{\boldsymbol X}]$ is isomorphic to $A\!\Join^{\sigma'} \!\! {\boldsymbol X}B[{\boldsymbol X}]$ (and $A+{\boldsymbol X}B[\![{\boldsymbol X}]\!]$ is isomorphic to $A\!\Join^{\sigma''}\!\! {\boldsymbol X}B[\![{\boldsymbol X}]\!]$).
Since we have the following canonical isomorphisms
$$
 \dfrac{B[{\boldsymbol X}]}{{\boldsymbol X}B[{\boldsymbol X}]}\cong B \cong \dfrac{B[\![{\boldsymbol X}]\!]}{{\boldsymbol X}B[\![{\boldsymbol X}]\!]}\,,
$$
in the present situation, the homomophism ${\breve{\sigma}'}: A \hookrightarrow B[{\boldsymbol X}]/{\boldsymbol
X}B[{\boldsymbol X}]$ (or, ${\breve{\sigma}''}:A \hookrightarrow B[\![{\boldsymbol X}]\!]/{\boldsymbol
X}B[\![{\boldsymbol X}]\!]$)  is finite.  Hence, statements (i) and (ii) follow easily from  Proposition
\ref{finite}.

 (i) (or, (ii)) $\Rightarrow$ (iii).
Assume that  $A+{\boldsymbol X}B[{\boldsymbol X}]$ (or, $A+{\boldsymbol X}B[\![{\boldsymbol X}]\!]$) is a
Noetherian ring.  By Proposition \ref{noet}, or by the isomorphism  $(A+{\boldsymbol X}B[{\boldsymbol
X}])/{\boldsymbol X}B[{\boldsymbol X}] \cong A$ (respectively $(A+{\boldsymbol X}B[\![{\boldsymbol
X}]\!])/{\boldsymbol X}B[\![{\boldsymbol X}]\!] \cong A$),  we deduce that $A$ is also a Noetherian ring.
Moreover,  by assumption, the ideal $I$ of $A+{\boldsymbol X}B[{\boldsymbol X}]$ (respectively, of $A+{\boldsymbol
X}B[\![{\boldsymbol X}]\!]$) generated by the set $\{bX_k \mid b\in B,1\leq k\leq n\}$ is finitely generated.
Hence $I=(f_1,f_2,  \z,  f_m)$,  for some $f_1,f_2, \z, f_m\in I$. Let $\{b_{jk}\mid 1\leq k\leq n\}$ be the set
of coefficients of linear monomials of the polynomial (respectively, power series) $f_j$,  $1 \leq j \leq m$. It
is easy to verify that $\{b_{jk} \mid 1\leq j\leq m,1\leq k\leq n\}$ generates $B$ as $A$--module; thus $A
\subseteq B$ is a finite ring extension.\end{proof}

\begin{remark}
Let $A\subseteq B$  be a ring extension, and let $X$ be  an  indeterminate over $B$. Note that the ideal $J' =XB[X]$ of $B[X]$ is never  finitely genera\-ted as an $A$--module (with the structure induced by the inclusion $\sigma':A\hookrightarrow B[X]$).   As a matter of fact, assume that $\{g_1,g_2, \z,g_r\} \ (\subset B[X])$  is a set of generators of   $J'$   as $A-$module and set $N:=\max\{{\rm deg}({{g_i}}) \mid i=1,2, \z,r\}$. Clearly,  we have $X^{N+1}\in  J'\setminus\sum_{i=1}^rAg_i$,  which is a contradiction.
Therefore, the previous observation shows that the Noetherianity of the ring $\da$ does not imply that $J$ is  finitely generated as an $A$--module  (with the structure induced by $f$); for instance $\mathbb R+X \mathbb C[X] \ (\cong \mathbb R\!\Join^{\sigma'}\!\!X\mathbb C [X]$, where $\sigma': \mathbb R\hookrightarrow \mathbb C[X]$ is the natural embedding) is a Noetherian ring (Proposition \ref{noegenerale}),  but $X \mathbb C[X]$ is not finitely generated as an $\mathbb R$--vector space.  This fact shows that condition (a) (or, equivalently, (b) or (c)) of Proposition \ref{noepiubello} is not necessary for the Noetherianity of $\da$.
\end{remark}
\begin{example}\label{noeancorapiugenerale} Let $A \subseteq B$ be a ring extension, $J$ an ideal of $B$ and ${\boldsymbol X}:=\{X_1,\z,X_r\}$ a finite set of intederminates over $B$. We set $B':=B[{\boldsymbol X}]$,   $J':= {\boldsymbol X}J[{\boldsymbol X}]$ and we denote by $\sigma'$ the canonical embedding of $A$ into $B'$.  By a routine argument, it is easy to see that 
 the ring $A\Join^{\sigma'} \!\!J'$ is naturally isomorphic to the ring $A+{\boldsymbol X}J[{\boldsymbol X}]$.  
Now, we want to show that, in this case, we can characterize the Noetherianity of the ring $A+{\boldsymbol X}J[{\boldsymbol X}]$, without assuming a finiteness
 condition on the inclusion $A \subseteq B$ (as in  Corollary \ref{noegenerale} (iii)) or on the inclusion $A+{\boldsymbol X}J[{\boldsymbol X}] \subseteq B[{\boldsymbol X}]$. More precisely,  
 \it  the following conditions are equivalent.
\begin{enumerate}
\item[\rm (i)] $A+{\boldsymbol X}J[{\boldsymbol X}]$ is a Noetherian ring.
\item[\rm (ii)] $A$ is a Noetherian ring, $J$ is an idempotent ideal of $B$ and it is finitely generated as an $A$--module.
\end{enumerate} \rm 
(i)$\Rightarrow$(ii). Assume that $R:=A+{\boldsymbol X}J[{\boldsymbol X}] = A+J'$ is a Noetherian ring. Then, clearly,  $A$ is  Noetherian, since $A$ is canonically isomorphic to  $R/J'$.
Now, consider the ideal $L$ of $R$ generated by the set  of linear monomials $\{bX_i\mid 1\leq i \leq r,\ b\in J \}$. By assumption, we can find   $\ell_1,\ell_2, \z, \ell_t\in L$ such that $L  = \sum_{k=1}^t \ell_k R$.   Note  that $\ell_k(0, 0, \z,0)= 0$, for all $k $, $1\leq k \leq t$.  If we denote by $b_k$ the coefficient of the monomial $X_1$ in the polynomial $\ell_k$,  then it is easy to see that $\{b_1,b_2, \z,b_t\}$ is a set of generators of $J$ as an  $A$--module.

The next step is to show that $J$ is an idempotent ideal of $B$. By assumption,  $J'$ is a \ec~finitely generated ideal of $R$. Let  $$g_h:=  \sum_{j_1+\z+j_r=1}^{m_h}c_{h, j_1\z j_r} X_1^{j_1}\cdots X_r^{j_r}, \mbox{ with }  h=1,2, \z, s, $$
 be a finite set of generators of $J'$ in $R$.
Set $\overline{j_1}:=\max\{j_1 \mid c_{h,j_10\z 0} \neq 0, \mbox{ for } 1\leq h  \leq s\}$. Take now an arbitrary element  $b\in J$ and consider the monomial $bX_1^{\overline{j_1}+1}\in J'$. Clearly, we have 
$$
bX_1^{\overline{j_1}+1}= \sum_{h=1}^s f_hg_h, \, \mbox{with } f_h:= \sum_{e_1+\z+e_r=0}^{n_h}d_{h, e_1\z e_r} X_1^{e_1}\cdots X_r^{e_r} \in R\,.
$$
 Therefore,  
 $$
 b=  \sum_{h=1}^s \;  \sum_{\; j_1+e_1=\overline{j_1}+1} c_{h, j_10\z 0} d_{h, e_10\z 0} \,.
 $$ 
 Since $j_1< \overline{j_1}+1$, we have necessarily that $e_1\geq 1$. Henceforth $f_h$ belongs to $J'$ and so $d_{h, e_10\z 0} \in  J$, for all $h$,  $1\leq h \leq s$. This proves that $b\in J^2$. 
 
(ii)$\Rightarrow$(i).  In this situation, by Nakayama's lemma, we easily deduce that $ J= eB$, for some idempotent element $e\in J$.   Let $\{b_1,\z,b_s\}$ be a set of generators of $J$ as an $A$--module, i.e., $J= eB = \sum_{1 \leq h \leq s}b_hA$.
We consider a new set of indeterminates over $B$ (and $A$) and precisely  ${\boldsymbol Y}:=
\{Y_{ih}\mid 1\leq i \leq r,  1\leq h \leq s\}$.  We can define a map $\varphi:A[{\boldsymbol X}, {\boldsymbol Y}]\rightarrow B[{\boldsymbol X}]$ by setting $\varphi(X_i) := eX_i$, and $\varphi(Y_{ih}) := b_hX_i$, for all $ i=1,\z,r,\ h=1,\z, s$.
It is easy to see that $\varphi$ is a ring homomorphism and Im$(\varphi) \subseteq R \ ( = A +{\boldsymbol X}J[{\boldsymbol X}]) $.  Conversely, let 
$$
f:=a+ \sum_{i=1}^r\left(\sum_{e_{i_1}+\z+e_{i_r}=0}^{n_i}c_{i, e_{i_1}\z e_{i_r}}X_1^{e_{i_1}}\cdots X_r^{e_{i_r}}\right)\!\!X_i\in R \; (\mbox{and so } c_{i, e_{i_1}\z e_{i_r}} \in J)\,.
$$ 
Since $J = \sum_{1 \leq h \leq s}b_hA$, then for all $i=1,\z, r$ and $e_{i_1},\z ,e_{i_r}$, with $e_{i_1}+\z+e_{i_r}\in\{0,\z, n_i\}$,  we can find  elements 
$a_{i, e_{i_1}\z e_{i_r}, h} \in A$, with $1 \leq h \leq s$,  such that $c_{i, e_{i_1}\z e_{i_r}}= \sum_{h=1}^s
a_{i, e_{i_1}\z e_{i_r}, h}b_h$.  Consider the polynomial
$$
g:=a+ \sum_{i=1}^r\; \sum_{h=1}^s \; \sum_{e_{i_1}+\z+e_{i_r}=0}^{n_i} a_{i, e_{i_1}\z e_{i_r}, h}X_1^{e_{i_1}}\cdots X_r^{e_{i_r}}Y_{ih} \in A[{\boldsymbol X}, {\boldsymbol Y}].
$$
It is straightforward to see that $\varphi(g)=f$ and so  Im$(\varphi) = R$. By Hilbert Basis Theorem, we conclude easily that $R$ is Noetherian.
\end{example}

\begin{remark} We preserve the { notation} of Example \ref{noeancorapiugenerale}.

{ (1)}  Note that in the previous example, when $J=B$, we reobtain Corollary \ref{noegenerale} ((i)$\Leftrightarrow$(iii)).  If $B=A$ and $I$ is an ideal of of $A$, then we simply have that   $A+{\boldsymbol X}
I[{\boldsymbol X}]$ \emph{is a Noetherian ring if and only if  $A$ is a Noetherian ring and $I$ is an idempotent ideal of $A$}.   Note the previous two cases were studied as separate cases by S. Hizem, who announced similar results in her talk at the Fez Conference in June 2008, { presenting an ample and systematic study of the transfer of various finiteness conditions  in the constructions  $A+{\boldsymbol X}I[{\boldsymbol X}]$ and $A+{\boldsymbol X}B[{\boldsymbol X}]$. } 

{ (2)} The Noetherianity of $B$ it is not a necessary condition for the Noetherianity of the ring $A+{\boldsymbol X}J[{\boldsymbol X}]$. For instance, take $A$ any field, $B$ the product of infinitely many copies of $A$, so that we can consider $A$ as a subring of $B$, via the diagonal ring embedding $a\mapsto (a,a,\z)$,  $a\in A$. Set $J:=(1,0,\z)B$. Then $J$ is an idempotent ideal of $B$ and, { at the same time, } a cyclic $A$-module. Thus, by Example \ref{noeancorapiugenerale}, $A+{\boldsymbol X}J[{\boldsymbol X}]$ is a Noetherian ring. Obviously, $B$ is not Noetherian. 

{ (3)}
Note that, 
 if $A+{\boldsymbol X}J[{\boldsymbol X}]$ is  Noetherian  and $B$ is not Noetherian, then 
$A \subseteq B$ and $A+{\boldsymbol X}J[{\boldsymbol X}] \subseteq B[{\boldsymbol X}]$ are necessa\-ri\-ly not finite. Moreover, it is easy to see that $A+{\boldsymbol X}J[{\boldsymbol X}] \subseteq B[{\boldsymbol X}]$ is a finite extension if and only if  the canonical homomorphism $A \hookrightarrow B[{\boldsymbol X}]/({\boldsymbol X}J[{\boldsymbol X}])$ is finite. Finally, it can be shown that last condition holds if and only if $J=B$ and $A\subseteq B$ is finite. 
\end{remark}


                              %

\end{document}